\documentclass[12pt,a4paper]{article}

\usepackage{a4wide}

\usepackage[latin1]{inputenc}
\usepackage[T1]{fontenc}
\usepackage[english]{babel}

\usepackage{amsmath}
\usepackage{amssymb}
\usepackage{amsfonts,makeidx,amsthm,mathrsfs}
\newcommand{\F}{\mathscr{F}}

\newcommand{\Lr}{\mathscr{L}}
\newcommand{\dvol}{\frac{\omega^n}{n!}}
\newcommand{\dvolphi}{\frac{\omega_{\phi}^n}{n!}}
\newcommand{\dsubvol}{\frac{\omega^{n-1}}{(n-1)!}}
\newcommand{\dsubsubvol}{\frac{\omega^{n-2}}{(n-2)!}}

\def\End{\mathop{\mathrm{End}}\nolimits}
\newcommand{\tr}{\mathrm{tr}}

\newcommand{\R}{\mathbb{R}}
\newcommand{\C}{\mathbb{C}}

\newcommand{\Z}{\mathbb{Z}}

\def\Diff{\mathop{\mathrm{Diff}}\nolimits}

\def\cal{\mathop{\mathrm{Cal}}\nolimits}

\def\dim{\mathop{\mathrm{dim}}\nolimits}
\def\ham{\mathop{\mathrm{Ham}}\nolimits}

\def\symp{\mathop{\mathrm{Symp}}\nolimits}

\def\diff{\mathop{\mathrm{Diff}}\nolimits}

\newcommand{\E}{\mathcal{E}}
\newcommand{\J}{\mathcal{J}}

\newcommand{\LC}{\mathrm{lc}}
\newcommand{\MOm}{\mathcal{M}_{\Theta}}
\newcommand{\grad}{\textrm{grad}}

\newcommand{\M}{\mathcal{M}}
\newcommand{\extwedge}{\stackrel{\circ}{\wedge}}

\newtheorem{theoremintro}{Theorem}
\newtheorem{theoremprinc}{Theorem}
\newtheorem{theorem}{Theorem}[section]

\newtheorem{lemme}[theorem]{Lemme}
\newtheorem{cor}[theorem]{Corollary}
\newtheorem{prop}[theorem]{Proposition}
\newtheorem{defi}[theorem]{Definition}

\theoremstyle{definition} 
\newtheorem{ex}[theorem]{Exemple}

\theoremstyle{remark}

\newtheorem{rem}[theorem]{Remark}

\begin{document}

\renewcommand{\refname}{R\'ef\'erences bibliographiques}

\title{Infinite dimensional moment map geometry and closed Fedosov's star products.\footnote{Work supported by the Belgian Interuniversity Attraction Pole (IAP) DYGEST}}

\author{\small La Fuente-Gravy Laurent \\
	\scriptsize{laurent.lafuente@uclouvain.be}\\
	\footnotesize{Institut de Recherche en Math\'ematique et Physique, Universit\'e catholique de Louvain}\\[-7pt]
	\footnotesize Chemin du cyclotron, 2, 1348 Louvain-la-Neuve, Belgium \\[-7pt]} 

\maketitle

\begin{abstract}
We study the Cahen-Gutt moment map on the space of symplectic connections of a symplectic manifold. Given a K\"ahler manifold $(M,\omega,J)$, we define a Calabi-type functional $\F$ on the space $\MOm$ of K\"ahler metrics in the class $\Theta:=[\omega]$.
We study the space of zeroes of $\F$. When $(M,\omega,J)$ has non-negative Ricci tensor and $\omega$ is a zero of $\F$, we show the space of zeroes of $\F$ near $\omega$ has the structure of a smooth finite dimensional submanifold. We give a new motivation, coming from deformation quantization, for the study of moment maps on infinite dimensional spaces. More precisely, we establish a strong link between trace densities for star products (obtained from Fedosov's type methods) and moment map geometry on infinite dimensional spaces. As a byproduct, we provide, on certain K\"ahler manifolds, a geometric characterization of a space of Fedosov's star products that are closed up to order $3$ in $\nu$.
\end{abstract}

\noindent {\footnotesize {\bf Keywords:} Symplectic connections, Moment map, Deformation quantization, closed star products, K\"ahler manifolds.}

\newpage

\tableofcontents

%%%%%%%%%%%%%%%%%%%%%%%%%%%%%%%%%%%%%%%%%%%%%%%%%%%%%%%%%%%%%%%%%%%%%%%%%%%%%%%%%%%%%%%%%%%%%%%%%%%%%%%%
%%%%%%%%%%%%%%%%%%%%%%%%%%%%%%%%%%%%%%%%%%%%%%%%%%%%%%%%%%%%%%%%%%%%%%%%%%%%%%%%%%%%%%%%%%%%%%%%%%%%%%%%

\section{Introduction}

The idea that curvature can be seen as moment map on various infinite dimensional spaces has been a great source of inspiration in mathematics, see Atiyah-Bott \cite{AB}, Donaldson \cite{Donald,Donald2}.

We give a new motivation, coming from deformation quantization \cite{BFFLS}, for the study of moment maps on infinite dimensional spaces. More precisely, we establish a strong link between trace densities for star products (obtained from Fedosov's type methods) and moment map geometry on infinite dimensional spaces. We study the Cahen-Gutt moment map on the space of symplectic connections \cite{cagutt}.
As a byproduct, we provide, on certain K\"ahler manifolds, a geometric characterization of a space of Fedosov's star products that are closed up to order $3$ in $\nu$.

Consider a symplectic manifold $(M,\omega)$. The space $\E(M,\omega)$ of symplectic connections (torsion-free connection $\nabla$ with $\nabla \omega=0$) on $(M,\omega)$ is in a natural way, an infinite dimensional symplectic manifold. The group $\ham(M,\omega)$ of Hamiltonian diffeomorphisms of $(M,\omega)$ acts symplectically on $\E(M,\omega)$. Cahen and Gutt \cite{cagutt} computed a moment map $$\mu:\nabla\in \E(M,\omega) \mapsto \mu(\nabla)\in C^{\infty}(M)$$ for this action (see Theorem \ref{theor:momentE} below). Fox \cite{Fox} defines the notion of extremal symplectic connections. Those are connections that are critical points of the squared $L^2$-norm of $\mu$.

%Now, to any connection $\nabla\in \E(M,\omega)$ one associates its corresponding Fedosov's star product $*_{\nabla,0}$. We observe that the moment map property of $\mu$ implies that $\mu(\nabla)$ is the first (potentially) non-trivial term of a trace density for $*_{\nabla,0}$. Fedosov gives in \cite{fed4} an algorithm to compute completely the trace density for $*_{\nabla,0}$. 

Let $(M,\omega,J)$ be a closed K\"ahler manifold and set $\Theta$ the K\"ahler class of $\omega$. Denote by $\MOm$ the space of K\"ahler forms in the class $\Theta$, it is a Fr\'echet (infinite dimensional) manifold. We define a Calabi type functional $\F$ on $\MOm$ that is the square of the $L^2$-norm of $\mu$ (where both the $L^2$-norm and $\mu$ are taken with respect to the point where we are in $\MOm$). Our Theorem \ref{theorprinc:zeroF} gives the structure of the space of zeroes of $\F$ in a neighbourhood of a K\"ahler metric which is a zero of $\F$ with Ricci tensor non-negative everywhere.

\begin{theoremintro}\label{theorprinc:zeroF}
Assume the closed K\"ahler manifold $(M,\omega,J)$ has a Ricci tensor which is non-negative everywhere. 

If $\omega$ is a zero of $\F$, then there exists an open neighbourhood $U$ of $\omega$ in $\MOm$ such that the only zeroes of $\F$ inside $U$ is the $H_0(M,J)$-orbit of $\omega$, where $H_0(M,J)$ is the reduced automorphism group of $(M,\omega,J)$.
\end{theoremintro}

Deformation quantization as defined in \cite{BFFLS} is a formal associative deformation of the Poisson algebra $(C^{\infty}(M),.,\{\cdot,\cdot\})$ of a Poisson manifold $(M,\pi)$ in the direction of the Poisson bracket. The deformed algebra is the space $C^{\infty}(M)[[\nu]]$ of formal power series of smooth functions with composition law $*$ called star product. The existence of star products was obtained first in the symplectic case by Dewilde-Lecomte \cite{DWL}, Fedosov \cite{fed2} and Omori-Maeda-Yoshioka \cite{OMY} and finally in the Poisson case by Kontsevitch \cite{Kon}.

Star products on symplectic manifolds admit a trace \cite{fed3,fed,NT,gr} that is a $\nu$-linear functional defined on formal series of smooth functions with compact support with values in $\R[[\nu]]$ that vanishes on the star commutators. A trace for a star product $*$ is determined by its trace density $\rho^{(*)}\in C^{\infty}(M)[\nu^{-1},\nu]]$ :
$$\tr^*(F)=\int_M F\rho^{(*)} \dvol.$$ 
Closed star products in the sense of Connes, Flato, Sternheimer \cite{CFS} are star products for which $\rho^{(*)}\equiv 1$ is a trace density up to order $\frac{\dim M}{2}$ in $\nu$. 

We consider star products on symplectic manifolds obtained by the geometric Fedosov's construction. On a symplectic manifold $(M,\omega)$, for any choice of symplectic connection $\nabla$ and formal power series $\Omega$ of closed $2$-forms the Fedosov's method produces a star product, denoted in the sequel by $*_{\nabla,\Omega}$. Our main observation is that $\mu(\nabla)$ is the first non-trivial term of a trace density for $*_{\nabla,0}$. 

The above Theorem \ref{theorprinc:zeroF} gives in fact the local structure of a space of ``natural'' Fedosov's star products on a K\"ahler manifold that are closed up to order $3$ in $\nu$. Let $(M,\omega,J)$ be a closed K\"ahler manifold with $\Theta:=[\omega]$. To any $\widetilde{\omega}\in \MOm$, one associates the Fedosov's star product $*_{\widetilde{\nabla},0}$ with $\widetilde{\nabla}$ the corresponding Levi-Civita connection. 

\begin{theoremintro} \label{theorprinc:*closed}
Let $(M,\omega,J)$ be a closed K\"ahler manifold with Levi-Civita connection $\nabla$ and Ricci tensor everywhere non negative. Assume the Fedosov's star product $*_{\nabla,0}$ is closed up to order $3$ in $\nu$.

Then, there exists an open neighbourhood $U$ of $\omega$ in $\MOm$ such that if $\widetilde{\omega}\in U$ gives rise
to the star product $*_{\widetilde{\nabla},0}$ closed up to order $3$ in $\nu$ then there is an $f\in H_0(M,J)$ inducing an isomorphism
$$f^*:(C^{\infty}(M)[[\nu]],*_{\widetilde{\nabla},0})\stackrel{\cong}{\rightarrow}(C^{\infty}(M)[[\nu]],*_{\nabla,0}).$$
\end{theoremintro}

Finally, we consider more general Fedosov's star products and the Wick star product obtained by Bordemann-Waldmann \cite{BW}. In both case, we observe that the first non-trivial term of a trace density can be interpreted as a  moment map on an infinite dimensional manifold.

%%%%%%%%%%%%%%%%%%%%%%%%%%%%%%%%%%%%%%%%%%%%%%%%%%%%%%%%%%%%%%%%%%%%%%%%%%%%%%%%%%%%%%%%%%%%%%%%%%%%%%%%
%%%%%%%%%%%%%%%%%%%%%%%%%%%%%%%%%%%%%%%%%%%%%%%%%%%%%%%%%%%%%%%%%%%%%%%%%%%%%%%%%%%%%%%%%%%%%%%%%%%%%%%%

\section*{Acknowledgment} We thank Simone Gutt and Michel Cahen for their encouragement and the interest they showed in our work. We thank Till Br\"onnle who introduced us to the moment map interpretation of the scalar curvature. We also desire to thank Joel Fine and Mehdi Lejmi for valuable discussions.

%%%%%%%%%%%%%%%%%%%%%%%%%%%%%%%%%%%%%%%%%%%%%%%%%%%%%%%%%%%%%%%%%%%%%%%%%%%%%%%%%%%%%%%%%%%%%%%%%%%%%%%%
%%%%%%%%%%%%%%%%%%%%%%%%%%%%%%%%%%%%%%%%%%%%%%%%%%%%%%%%%%%%%%%%%%%%%%%%%%%%%%%%%%%%%%%%%%%%%%%%%%%%%%%%

\section{Symplectic connections} \label{sect:symplconn}

%%%%%%%%%%%%%%%%%%%%%%%%%%%%%%%%%%%%%%%%%%%%%%%%%%%%%%%%%%%%%%%%%%%%%%%%%%%%%%%%%%%%%%%%%%%%%%%%%%%%%%%%

Consider a symplectic manifold $(M,\omega)$ of dimension $2n$. A symplectic connection $\nabla$ on $(M,\omega)$ is a torsion-free connection such that $\nabla \omega=0$. There always exists a symplectic connection on a symplectic manifold but this connection is not unique.
Indeed, given a symplectic connection $\nabla$ on $(M,\omega)$, 
any other symplectic connection on $(M,\omega)$ is of the form
$$\nabla'_X Y = \nabla_X Y + A(X)Y$$
where $A$ is a section of $\Lambda^1 (M)\otimes End(TM,\omega)$ satisfying $A(X)Y-A(Y)X=0$. Equivalently, $\omega(A(X)Y,Z)$ is a completely symmetric 
$3$-tensor, i.e. a section of $S^3T^*M$. Then, the space $\E(M,\omega)$ of symplectic connections is the affine space
$$\E(M,\omega)=\nabla+ \Gamma(S^3T^*M) \textrm{ for some } \nabla \in \E(M,\omega).$$

We assume from now that $(M,\omega)$ is closed. Then, there is a natural symplectic form on $\E(M,\omega)$. For $A, B \in T_{\nabla} \E(M,\omega)$, seen as elements of $\Lambda^1 (M)\otimes End(TM,\omega)$, one defines
$$\Omega^{\E}_{\nabla}(A,B):= 2\int_M \tr(A\extwedge B)\wedge \dsubvol= -\int_M \Lambda^{kl}\tr(A(e_k)B(e_l))\frac{\omega^n}{n!},$$
where $\extwedge$ is the product on $\Lambda^1 (M)\otimes End(TM,\omega)$ induced by the usual $\wedge$-product on forms and the composition on the endomorphism part, $\Lambda^{kl}$ is defined by $\Lambda^{kl}\omega_{lt}=\delta^k_t$ for $\omega_{lt}:=\omega(e_l,e_t)$ for a frame $\{e_k\}$ of $T_xM$. 
The $2$-form $\Omega^{\E}$ is a symplectic form on $\E(M,\omega)$. 

\begin{rem}
The constant $2$ appearing in the definition of $\Omega^{\E}$ is introduced to fit with the form defined in \cite{cagutt}. Indeed Cahen and Gutt defined $\Omega^{\E}$ for $A,B\in T_{\nabla} \E(M,\omega)$, seen as element of $\Gamma(S^3T^*M)$, by 
$$\Omega^{\E}_{\nabla}(A,B):=\int_M \Lambda^{i_1j_1}\Lambda^{i_2j_2}\Lambda^{i_3j_3}A_{i_1i_2i_3}B_{j_1j_2j_3}\frac{\omega^n}{n!}.$$
\end{rem}

A positive almost complex structure $J$ compatible with $\omega$
induces a positive almost complex structure $J^{\E}$ compatible with $\Omega^{\E}$ defined by
\begin{equation} \label{eq:JE} 
\left(J^{\E}(A) \right)(X)Y:= -JA(JX)JY,\ \  \forall A\in T_{\nabla}\E(M,\omega).
\end{equation}
And, as usual, one defines the associated Riemannian metric $G^{J^{\E}}$ on $\E(M,\omega)$ by 
\begin{equation} \label{eq:GJE}
G^{J^{\E}}_{\nabla}(A,B):=\Omega_{\nabla}(A,J^{\E}B)=\int_M g^{kl}\tr(JA(e_k)JB(e_l))\frac{\omega^n}{n!}.
\end{equation}

There is a natural symplectic action of the group of symplectic diffeomorphisms on $\E(M,\omega)$. For $\varphi$, a symplectic
diffeomorphism, we define an action
\begin{equation} \label{eq:action}
(\varphi.\nabla)_X Y := \varphi_*(\nabla_{\varphi^{-1}_* X}\varphi^{-1}_* Y),
\end{equation}
for all $X,Y \in TM$ and $\nabla \in \E(M,\omega)$. 

Recall that a Hamiltonian vector field is a vector field $X_F$ for $F\in C^{\infty}(M)$ such that
$$i(X_F)\omega=dF.$$
We denote by $\ham(M,\omega)$ the group of Hamiltonian diffeomorphisms of the symplectic manifold $(M,\omega)$. The group $\ham(M,\omega)$ is a Lie group with Lie algebra the space $C^{\infty}_0(M)$ of normalised smooth functions \cite{ban2}, i.e. smooth functions $F$ such that $\int_M F \dvol=0$. 

The action defined in Equation (\ref{eq:action}) restricts to an action of the group $\ham(M,\omega)$. Let $X_F$ be a Hamiltonian vector field with $F\in C^{\infty}_0(M)$, the fundamental vector field on $\E(M,\omega)$ associated to this action is
\begin{equation*}
(X_F)^{*\E}(Y)Z:=\frac{d}{dt}|_0\varphi_{-t}^F.\nabla=(\Lr_{X_F}\nabla)(Y)Z=\nabla^2_{(Y,Z)}X_F + R^{\nabla}(X_F,Y)Z,
\end{equation*}
where $R^{\nabla}(U,V)W:=[\nabla_U,\nabla_V]W-\nabla_{[U,V]}W$ is the curvature tensor of $\nabla$.
A moment map for the action of $\ham(M,\omega)$ on $\E(M,\omega)$ is a map $\tilde{\mu}:\E(M,\omega) \rightarrow C^{\infty}(M)$, where $C^{\infty}(M)$ is viewed, using the $L^2$-product, as a subspace of the dual of $C^{\infty}_0(M)$, satisfying
\begin{equation} \label{eq:momentmu}
\frac{d}{dt}|_{0} \int_M \tilde{\mu}(\nabla+tA)F\dvol=\Omega^{\E}_{\nabla}((X_F)^{*\E},A)=\Omega^{\E}_{\nabla}(\Lr_{X_F}\nabla,A),
\end{equation}
where $\nabla \in \mathcal{E}(M,\omega) $ and $A \in T_{\nabla}\mathcal{E}(M,\omega)$. 

Denote by $Ric^{\nabla}$ the Ricci tensor of $\nabla$ defined by
$$Ric^{\nabla}(X,Y):=\tr[V\mapsto R^{\nabla}(V,X)Y]$$
for all $X,Y \in TM$. 

\begin{theorem} [Cahen-Gutt \cite{cagutt}] \label{theor:momentE}
The map $\tilde{\mu}:\mathcal{E}(M,\omega) \rightarrow C^{\infty}(M)$ defined by
$$\tilde{\mu}(\nabla):=(\nabla^2_{(e_p,e_q)} Ric^{\nabla})(e^{p},e^q) + P(\nabla)$$
where $\{e_k\}$ is a frame of $T_{x}M$, $\{e^l\}$ is the symplectic dual frame of $\{e_k\}$ (that is $\omega(e_k,e^l)=\delta_k^l$) and $P(\nabla)$ is the function defined by $P(\nabla)\dvol:=\frac{1}{2}\tr(R^{\nabla}(.,.)\extwedge R^{\nabla}(.,.))\wedge \dsubsubvol$, is a moment map for the action of $\ham(M,\omega)$ on $\E(M,\omega)$.
\end{theorem}

\begin{rem}
It is possible to extend the above theorem to non compact symplectic manifolds. Of course, the symplectic
form $\Omega^{\E}$ is only defined on $L^2$-section of $S^3T^*M$ for the $L^{2}$-norm induced by $G^{J^\E}$. 
\end{rem}

\begin{rem}
On a closed symplectic manifold, since $(\nabla^2_{(e_p,e_q)} Ric^{\nabla})(e^{p},e^q)$ has vanishing integral
the constant $\mu_0:=\int_{M} \tilde{\mu}(\nabla) \dvol$ is a topological constant. Hence, $\mu_0$ does not depend on the symplectic conection.
\end{rem}

\begin{rem}
One can develop the expression of $P(\nabla)$ and obtain
$$\tilde{\mu}(\nabla):=(\nabla^2_{(e_p,e_q)} Ric^{\nabla})(e^{p},e^q)  -\frac{1}{2} Ric^{\nabla}_{pq}(Ric^{\nabla})^{pq} + \frac{1}{4} R^{\nabla}_{pqrs}(R^{\nabla})^{pqrs}$$
where $R^{\nabla}_{pqrs}:=\omega(R^{\nabla}(e_p,e_q)e_r,e_s)$ for a frame $\{e_k\}$ of $T_{x}M$ and $(R^{\nabla})^{pqrs}:=\omega(R^{\nabla}(e^p,e^q)e^r,e^s)$ for the symplectic dual frame $\{e^l\}$ of $\{e_k\}$.
\end{rem}

\begin{defi}
On a closed symplectic manifold $(M,\omega)$, we will call the {\bf (normalized) moment map} the map $\mu$ defined by
$$\mu : \E(M,\omega) \rightarrow C^{\infty}_0(M): \nabla \mapsto \tilde{\mu}(\nabla)-\mu_0.$$
\end{defi}

We are interested in zeroes of the normalised moment map $\mu$ when the symplectic manifold is K\"ahler.
On a K\"ahler manifold $(M,\omega,J)$, one has a particular choice of symplectic connection : the Levi-Civita connection $\nabla$. Using the second Bianchi identity, one has $\Lambda^{kl}(\nabla_{e_l} Ric^{\nabla})(e_k,\cdot)=-\frac{1}{2} dScal^{\nabla}(J\cdot)$ where $Scal^{\nabla}$ is the scalar curvature. The moment map reduces to 
$$\mu(\nabla)=-\frac{1}{2}\Delta Scal^{\nabla}  -\frac{1}{2} Ric^{\nabla}_{pq}(Ric^{\nabla})^{pq} + \frac{1}{4} R^{\nabla}_{pqrs}(R^{\nabla})^{pqrs}-\mu_0.$$
Of course the connection is fixed in $\E(M,\omega)$. But one can vary other datas of the K\"ahler manifold such as the complex structure or the K\"ahler form. The advantage is that the space of complex structures compatible with $\omega$ and the space of
 K\"ahler forms in the same K\"ahler class as $\omega$ are ``smaller'' (still infinite dimensional) than $\E(M,\omega)$.

\noindent Let us give some examples of zeroes of $\mu$ on K\"ahler manifolds.

\begin{ex}
Assume $(M,\omega,J)$ is a K\"ahler manifold of dimension $2$. As pointed out by D.J. Fox \cite{Fox}, the moment map evaluated at the Levi-Civita connection $\nabla$ reduces to
$$\mu(\nabla)=-\frac{1}{2} \Delta Scal^{\nabla}.$$
So that $\mu(\nabla)=0$ if and only if the scalar curvature is constant.
\end{ex} 

\begin{ex}
On closed K\"ahler manifolds with constant holomorphic sectionnal curvature, the Levi-Civita connection provides a zero of the moment map.
\end{ex}

In \cite{Fox}, Fox studies the functional on $\E(M,\omega)$ defined by the $L^2$-norm of $\mu$. He describes its critical points when the symplectic manifold has dimension $2$.

In Section \ref{sect:cal}, we consider a K\"ahler manifold and define a similar functional on the space of K\"ahler metrics in the same K\"ahler class and study its zeroes.

%%%%%%%%%%%%%%%%%%%%%%%%%%%%%%%%%%%%%%%%%%%%%%%%%%%%%%%%%%%%%%%%%%%%%%%%%%%%%%%%%%%%%%%%%%%%%%%%%%%%%%%%
%%%%%%%%%%%%%%%%%%%%%%%%%%%%%%%%%%%%%%%%%%%%%%%%%%%%%%%%%%%%%%%%%%%%%%%%%%%%%%%%%%%%%%%%%%%%%%%%%%%%%%%%

\section{The space $\J(M,\omega)$}

%%%%%%%%%%%%%%%%%%%%%%%%%%%%%%%%%%%%%%%%%%%%%%%%%%%%%%%%%%%%%%%%%%%%%%%%%%%%%%%%%%%%%%%%%%%%%%%%%%%%%%%%

In this section, we consider a closed K\"ahler manifold $(M,\omega, J)$.
We study the Levi-Civita map $\LC:\J_{int}(M,\omega) \rightarrow \E(M,\omega)$ where $\J_{int}(M,\omega)$ is the space
of integrable complex structures on $M$ compatible with $\omega$ and $\LC(\widetilde{J}):=\nabla^{\widetilde{J}}$ is the Levi-Civita connection of the metric
$g_{\widetilde{J}}(.,.):=\omega(.,\widetilde{J}.)$ with $\widetilde{J}$ integrable. 

The goal of this section is to obtain conditions implying a certain non-degeneracy of $\LC^*\Omega^{\E}$. This is a technical condition that will show up in our study of zeroes of $\mu$ in Section \ref{sect:cal} (see Proposition \ref{prop:Dphi})

%%%%%%%%%%%%%%%%%%%%%%%%%%%%%%%%%%%%%%%%%%%%%%%%%%%%%%%%%%%%%%%%%%%%%%%%%%%%%%%%%%%%%%%%%%%%%%%%%%%%%%%%

\subsection{The K\"ahler structure on $\J(M,\omega)$} \label{subsect:JMomega}

%%%%%%%%%%%%%%%%%%%%%%%%%%%%%%%%%%%%%%%%%%%%%%%%%%%%%%%%%%%%%%%%%%%%%%%%%%%%%%%%%%%%%%%%%%%%%%%%%%%%%%%%

First, let us describe the space $\J(M,\omega)$ and its tangent space.

\begin{defi}
Let $(M,\omega)$ be a symplectic manifold. A {\bf compatible almost complex structure $J$} on $(M,\omega)$ 
is a section of $\End(TM)$ such that $J^2=-Id$, $\omega(J.,J.)=\omega(.,.)$ and $\omega(.,J.)$ is a Riemannian metric.
We set $\J(M,\omega)$ the {\bf space of all almost complex structures compatible with $\omega$}.
\end{defi}

%Because compatible almost complex structures are sections of the bundle of endomorphisms of $TM$,
%the difference $J-J'=A$, where $J,J' \in \J(M,\omega)$, is in $\Gamma(\End(TM))$. Then, $(J'+A)^2=-Id$,
%that is
%\begin{equation*}
%AJ'+J'A=-A^2.
%\end{equation*}
%And, because $\omega(X,JY)=\omega(Y,JX)$ (resp. for $J'$), we have
%\begin{equation*}
%\omega(X,AY)+\omega(AX,Y)=0.
%\end{equation*}
%This means that the tensor $\omega(.,A.)$ is symmetric.

\noindent The tangent space to $\J(M,\omega)$ at a point $J$ is
$$T_{J}\J(M,\omega)= \{A\in End(TM,\omega) \textrm{ s.t. } AJ+JA=0\}.$$

The space $\J(M,\omega)$ admits a K\"ahler structure. For $A,B \in T_{J}\J(M,\omega)$, we define a metric $G_J$
and a complex structure $I_J$ by : 
$$G^{\J}_J(A,B):=\int_M \tr(AB)\dvol \textrm{ and } I_J(A):=JA. $$
The associated K\"ahler form is
$$\Omega^{\J}_J(A,B):= G^{\J}_J(JA,B)=\int_M \tr(JAB) \dvol.$$

The Hamiltonian diffeomorphisms group acts symplectically on $\J(M,\omega)$ by 
$$\varphi.J:=\varphi_*\circ J \circ \varphi^{-1}_* \textrm{ for all } \varphi\in \ham(M,\omega).$$ 
Donaldson \cite{Donald} and Fujiki \cite{Fu} showed that this action admits a moment map given by the Hermitian scalar curvature. We use this moment map picture in the final section so we briefly recall its construction.

The fundamental vector fields for this action are given by
$$(X_F)^{*\J}:=\Lr_{X_F}J,\textrm{ for } F\in C^{\infty}_0(M).$$
Given a compatible $J\in \J(M,\omega)$, consider the Chern connection $\nabla$ on $(TM,J)$. It induces a complex connection $\nabla^{\Lambda^n}$ on $\Lambda^n (TM,J)$ (recall that $n=\frac{\dim M}{2}$). Denote by $R^{(n)}$ the curvature of $\nabla^{\Lambda^n}$. The Hermitian scalar curvature is the function $Scal(J)$ (depending on $J$) defined by
$$Scal(J) \dvol:= 2. R^{(n)}\wedge \dsubvol. $$
The Hermitian scalar curvature is a moment map for the action of $\ham(M,\omega)$ on $\J(M,\omega)$, \cite{Donald} and \cite{Fu}. That is, if $A:=\frac{d}{dt}|_0 J_t \in T_J\J(M,\omega)$, then for all $F\in C_0^{\infty}(M)$ we have :
\begin{equation} \label{eq:scalmoment} 
\Omega^{\J}_J((X_F)^{*\J},A)=\frac{d}{dt}|_0 \int_M F Scal(J_t)\dvol 
\end{equation}

\begin{defi}
We denote by $\J_{int}(M,\omega)$ the subspace of integrable complex structures in $\J(M,\omega)$.
\end{defi}

\noindent We describe the tangent space of $\J_{int}(M,\omega)$.

\begin{prop} \label{prop:nijenderiv}
Let $J\in \J_{int}(M,\omega)$, if $A \in T_J \J(M,\omega)$ is tangent to $\J_{int}(M,\omega)$ then the $2$-tensor
$$J(\nabla A (X)Y)-(\nabla A)(JX)Y$$
is symmetric in $X, Y$.
\end{prop}

\begin{proof}
Consider $J_t\in  \J_{int}(M,\omega)$ such that $\frac{d}{dt}|_0J_t=A$. 
Because each $J_t$ is integrable, the Nijenhuis tensor $N_{J_t}=0$. Differentiating this equation
at $t=0$ gives :
$$J(\nabla A (X)Y)-(\nabla A)(JX)Y=J(\nabla A (Y)X)-(\nabla A)(JY)X.$$
Then, the vector fields valued $2$-tensor  $J\nabla A(X)Y-(\nabla A)(JX)Y$ is symmetric.
\end{proof}

%%%%%%%%%%%%%%%%%%%%%%%%%%%%%%%%%%%%%%%%%%%%%%%%%%%%%%%%%%%%%%%%%%%%%%%%%%%%%%%%%%%%%%%%%%%%%%%%%%%%%%%%

\subsection{The map $\LC:\J_{int}(M,\omega) \rightarrow \E(M,\omega)$} \label{subsect:lcmap}

%%%%%%%%%%%%%%%%%%%%%%%%%%%%%%%%%%%%%%%%%%%%%%%%%%%%%%%%%%%%%%%%%%%%%%%%%%%%%%%%%%%%%%%%%%%%%%%%%%%%%%%%

We define the {\bf Levi-Civita map} to be the map 
$$\LC:\J_{int}(M,\omega) \rightarrow \E(M,\omega): J \mapsto \nabla^{J}$$
which associates to an integrable complex structure $J$ compatible with $\omega$, the Levi-Civita 
connection $\nabla^{J}$ of the K\"ahler metric $g(.,.):= \omega(.,J.)$. 

\begin{lemme} \label{lemme:LCequiv}
The map $\LC$ is equivariant with respect to the group of symplectic diffeomorphisms of $(M,\omega)$.
That is : 
$$\LC(\varphi.J)=\varphi.\LC(J)$$
for all $\varphi \in \symp(M,\omega)$ and $J\in \J_{int}(M,\omega)$.
\end{lemme}

\begin{proof}
It is a straightforward computation that $\nabla$ is the Levi-Civita connection of the metric
$g_J(\cdot,\cdot):= \omega(\cdot,J\cdot)$ if and only if $\varphi.\nabla$ is the Levi-Civita connection
of the metric $g_{\varphi.J}(\cdot,\cdot):= \omega(\cdot,\varphi.J\cdot)$.
\end{proof}

In the sequel of the paper, we will be concerned by K\"ahler manifolds satisfying a certain non-degeneracy condition.

\begin{defi} \label{def:condC}
A K\"ahler manifold $(M,\omega,J)$ is said to satisfy {\bf Condition C} if 
for $F\in  C^{\infty}_0(M)$ : 
$$ (\LC^{*} \Omega^{\E})_J\left(\Lr_{X_H}J,J\Lr_{X_F}J\right)=0, \forall H \in  C^{\infty}_0(M) \Rightarrow \Lr_{X_F}J=0. $$
\end{defi}

\begin{rem} \label{rem:LJLN}
The condition $\Lr_{X_F}J=0$ is equivalent to $\Lr_{X_F}\nabla=0$ (i.e. $\LC_{*J}(\Lr_{X_F}J)=0$).
\end{rem}

We now exhibit a geometric condition in term of the Ricci tensor of the K\"ahler manifold which implies the above Condition C.
For this, we study carefully the differential  of the Levi-Civita map : $\LC_{*J}:T_J \J_{int}(M,\omega) \rightarrow T_{\LC(J)}\E(M,\omega)$.

\begin{prop} \label{prop:LC*}
Let $A\in T_J \J_{int}(M,\omega)$ and write $B\in T_{\nabla}\E(M,\omega)$ such that $B=\LC_{*J}(A)$.

Then :  
\begin{enumerate}
\item $B$ is the unique solution to the equation
\begin{equation*} \label{eq:BJlin}
B(X)Y+JB(X)JY=- \nabla J A(X)Y.
\end{equation*}
\item if $JA \in T_J \J_{int}(M,\omega)$, then :
$$\LC_{*J}(JA)(X)Y=-J^{\E}B(X)Y +\frac{1}{2}\left(J(\nabla A) (JX)Y)+(\nabla A)(X)Y\right).$$
\end{enumerate}
\end{prop}

\begin{proof}
Consider $J_t$ in $\J_{int}(M,\omega)$ with $\frac{d}{dt}|_0J_t=A$.
Set $\nabla_t:= \LC(J_t)$, we want to compute $B:=\frac{d}{dt}|_0 \LC(J_t)$.
By definition, $\nabla_t J_t=0$. Differentiating this relation at $t=0$ gives :
\begin{equation} \label{eq:LC*}
0 =  \frac{d}{dt}|_0 (\nabla_t J_t)(X)Y  =  B(X)JY-J(B(X)Y)+(\nabla A)(X)Y.
\end{equation}
Suppose that there are two solutions $B$ and $B'$ of (\ref{eq:LC*}). The difference $C:=B-B'$ satisfies $C(X)JY-J(C(X)Y)=0$.
It means that the connection $\nabla+C$ preserves the complex structure $J$. Moreover, $\nabla+C$ is symplectic.
So that, $\nabla$ and $\nabla+C$ are torsion free connections preserving $g(\cdot,\cdot)=\omega(\cdot,J\cdot)$.
It implies $C=0$. 

For the second statement, set 
$$\tilde{B}(X)Y:= -J^{\E}B(X)Y +\frac{1}{2}\left(J(\nabla A) (JX)Y)+(\nabla A)(X)Y\right).$$
By Proposition \ref{prop:nijenderiv}, $\tilde{B}(X)Y$ is symmetric in $X$ and $Y$. One checks by direct computations that $\tilde{B}(X)$ sits in $End(TM,\omega)$. Moreover,
$$J\tilde{B}(X)Y-\tilde{B}(X)JY=\nabla(JA)(X)Y.$$
It implies $\tilde{B}=\LC_{*J}(JA)$.
\end{proof}

\begin{rem}
The formula of point 2 in the above Proposition \ref{prop:LC*} gives the defect for $\LC_{*J}$ to be complex anti-linear for the complex structures $I_J$ on $T_J\J_{int}(M,\omega)$ and $J^{\E}$ on $T_{\LC(J)}\E(M,\omega)$.
\end{rem}

\begin{prop} \label{prop:omeganegative}
Let $(M,\omega,J)$ be a K\"ahler manifold. If the Ricci tensor $Ric^{\nabla}$ is everywhere non negative (i.e. $\forall x\in M : Ric^{\nabla}(X_x,X_x)\geq 0\  \forall X_x \in T_xM$), then $(M,\omega,J)$ satisfies the Condition C.
\end{prop}

\begin{proof}
Consider $A$ and $JA \in T_J \J_{int}(M,\omega)$. We show that $\Omega^{\E}_{\LC(J)}(\LC_{*J}(A),\LC_{*J}(JA))\leq0$ with equality if and only if $\LC_{*J}(A)=0$.

Write $B=\LC_{*J}(A)$ and $\tilde{B}:=\LC_{*J}(JA)$. By Proposition \ref{prop:LC*}, we know the $J$-anti-linear part of $B$ and $\tilde{B}$. We write 
\begin{eqnarray} 
B(X)Y & = & \frac{1}{2}\left(-\nabla(JA)(X)Y\right)+B(X)^JY, \label{eq:B} \\
\tilde{B}(X)Y & = & \frac{1}{2} \left((\nabla A)(X)Y\right)+\tilde{B}(X)^JY. \label{eq:tildeB}
\end{eqnarray}
where $B(X)^J$ and $\tilde{B}(X)^J$ denotes the $J$-linear part of the fields of endomorphisms $B(X)$ and $\tilde{B}(X)$.
Again by Proposition \ref{prop:LC*}, one checks that $\tilde{B}(X)^J=-B(JX)^J$. We can now compute 
\begin{eqnarray*}
(\LC^*\Omega^{\E})_J(A,JA) & = & \Omega^{\E}_{\LC(J)}(B,\tilde{B}), \\
 & = & \frac{1}{4} \int_M \Lambda^{kl} \tr((\nabla JA)(e_k)(\nabla A(e_l))) \dvol + \int_M \Lambda^{kl} \tr(B(e_k)^JB(Je_l)^J) \dvol, \\
 & = & \frac{1}{4} \int_M \Lambda^{kl} \tr((\nabla JA)(e_k)(\nabla A)(e_l)) \dvol + \int_M g^{kl} \tr(B(e_k)^JB(e_l)^J) \dvol.
\end{eqnarray*}
The second term of the last equality is negative definite because the trace induces a negative definite bilinear form on $J$-linear endomorphisms of $(TM,\omega)$. It remains to deal with the first term.
\begin{eqnarray*}
\int_M \Lambda^{kl} \tr((\nabla JA)(e_k)(\nabla A(e_l)))\dvol & = & \int_M  \tr(JA\Lambda^{lk}\imath(e_l)\nabla_{e_k}(\nabla A))\dvol, \\
& = & \frac{1}{2} \int_M  \tr(JA\Lambda^{lk}R(e_k,e_l)( A))\dvol, \\
& = &   \int_M  \tr(JA\rho A)\dvol- \int_M  \tr(JA A\rho)\dvol \\
& = & 2\int_M  \tr(JA\rho A)\dvol,
\end{eqnarray*}
where $\rho\in \End(TM,\omega)$ denotes the Ricci endomorphism defined by $\omega(\rho X,Y):=Ric(X,Y)$.
One has
$$\tr(JA\rho A)=\Lambda^{lk}\omega(JA\rho Ae_k,e_l)=\Lambda^{lk}\omega(JAe_l,\rho Ae_k)=-2g^{kt}Ric(Ae_k,Ae_t)\leq 0.$$
So that $\Omega^{\E}(\LC_{*J}(A),\LC_{*J}(JA))\leq 0$. Moreover, if $\Omega^{\E}(\LC_{*J}(A),\LC_{*J}(JA))=0$, then $B(X)^J=0$ and also $\tilde{B}(X)^J=0$ for all $X$. But, this implies the tensor 
$(\nabla A)(X)Y$ to be symmetric in $X,Y$. Hence, we have
\begin{equation*}
0=(\LC^*\Omega^{\E})_J(A,JA)= \frac{1}{4} \int_M g^{kl} \tr((\nabla A)(e_k)(\nabla A(e_l))).
\end{equation*}
Because the trace is positive definite on $J$-anti-linear endomorphisms, it implies $\nabla A=0$ and thus $\LC_{*J}(A)=0$.

Now, we have to conclude that $(M,\omega,J)$ satisfies the Condition C. We consider $A=\Lr_{X_F}J$ for $F\in C^{\infty}_0(M)$ such that $\Omega^{\E}_{\LC(J)}(\LC_{*J}(\Lr_{X_H}J),\LC_{*J}(J\Lr_{X_F}J))=0$ for all $H\in C^{\infty}_0(M)$. Taking $H=F$, we showed above that necessarily $\LC_{*J}(\Lr_{X_F}J)=0$.
By the equivariance of the map $\LC$, we have $\Lr_{X_F}\nabla =0$. Then, by remark \ref{rem:LJLN}, it means $\Lr_{X_F}J=0$.
\end{proof}

\begin{lemme} \label{lemme:Iinv}
The $2$-form $\LC^*\Omega^{\E}$ is $I$-invariant, $I$ being the complex structure on $\J_{int}(M,\omega)$. 
That is if $A, A'$ and $JA, JA'\in T_J\J_{int}(M,\omega)$ then
$$(lc^*\Omega^{\E})_J(JA,JA')=(lc^*\Omega^{\E})_J(A,A').$$
\end{lemme}

\begin{proof}
From the equations (\ref{eq:B}) and (\ref{eq:tildeB}), one checks directly that $\LC^*\Omega^{\E}$ is $I$-invariant.
\end{proof}

\begin{rem}
By the Lemma \ref{lemme:Iinv}, the Condition C is equivalent to :
for $F\in  C^{\infty}_0(M)$ : 
$$ (\LC^{*} \Omega^{\E})_J\left(\Lr_{X_F}J,J\Lr_{X_H}J\right)=0, \forall H \in  C^{\infty}_0(M) \Rightarrow \Lr_{X_F}J=0. $$
\end{rem}

\begin{rem} 
The Proposition \ref{prop:omeganegative} (and its proof) and the Lemma \ref{lemme:Iinv} means that :
if $(M,\omega,J)$ has Ricci tensor positive definite at all points of $M$, then the $2$-form $(\LC^*\Omega^{\E})_J$ restricted on the subspace 
$$\mathcal{T}:=\{\Lr_{X_F}J+J\Lr_{X_H}J \ | \ F, H \in C^{\infty}_0(M)\} \leq T_J\J_{int}(M,\omega)$$
is a symplectic form and $I_J$ is a compatible negative complex structure with $(\LC^*\Omega^{\E})_J$.
\end{rem}

\noindent The Condition C is stable by complex transformation.

\begin{prop} \label{prop:condCcplx} 
Assume the closed K\"ahler manifold $(M,\omega,J)$ satisfies the Condition C. Let $f$ be a diffeomorphism of $M$ preserving $J$ (i.e. $f_*\circ J=J\circ f_*$), then the K\"ahler manifold $(M,f^*\omega,J)$ also satisfies the Condition C.
\end{prop}

\begin{proof}
Let $\tilde{J} \in \J_{int}(M,\omega)$, define $f^{-1}.\tilde{J} :=f^{-1}_*\circ \tilde{J} \circ f_*$. Then, $f^{-1}.\tilde{J}\in \J_{int}(M,f^*\omega)$. Because it is clearly invertible, $f^{-1}$ induces a bijection $f^{-1}.:\J_{int}(M,\omega)\rightarrow \J_{int}(M,f^*\omega)$. Its differential acts on $A\in T_J\J_{int}(M,\omega)$ by $(f^{-1}.)_{*J}(A)=f^{-1}_*\circ A \circ f_*$.
So that, using $f^{-1}.J=J$ we have for $F\in C^{\infty}(M)$ : 
$$(f^{-1}.)_{*J}(\Lr_{X^{\omega}_F}J)=\Lr_{f^{-1}_*X^{\omega}_F}J=\Lr_{X^{f^*\omega}_{f^*F}}J,$$
where we keep track of the symplectic form with respect to the Hamiltonian vector fields are taken from.

In the same vein, $f^{-1}$ induces a symplectomorphism 
$$f^{-1}. : (\E(M,\omega),\Omega^{\E(M,\omega)}) \rightarrow (\E(M,f^*\omega),\Omega^{\E(M,f^*\omega)}),$$
defined by $(f^{-1}.\nabla)_X Y := f^{-1}_*(\nabla_{f_*X} f_*Y)$. Moreover, for $\tilde{J}$ in $\J_{int}(M,\omega)$,
one checks that $f^*\omega(\cdot,f^{-1}.\tilde{J}\cdot)=f^*g^{\tilde{J}}(\cdot,\cdot)$. So that, $f^{-1}.$ commutes with the Levi-Civita maps :
$$f^{-1}.\circ \LC = \LC' \circ f^{-1}.$$
where $\LC' : \J_{int}(M,f^*\omega)\rightarrow \E(M,f^*\omega)$ is the Levi-Civita map.

All this proves that 
$$ (\LC'^{*} \Omega^{\E(M,f^*\omega)})_J\left(\Lr_{X^{f^*\omega}_{f^*F}}J,J\Lr_{X^{f^*\omega}_{f^*H}}J\right)=(\LC^{*} \Omega^{\E(M,\omega)})_J\left(\Lr_{X_F}J,J\Lr_{X_H}J\right).$$
Then $(M,f^*\omega,J)$ also satisfies the Condition C, as desired.
\end{proof}

%%%%%%%%%%%%%%%%%%%%%%%%%%%%%%%%%%%%%%%%%%%%%%%%%%%%%%%%%%%%%%%%%%%%%%%%%%%%%%%%%%%%%%%%%%%%%%%%%%%%%%%%
%%%%%%%%%%%%%%%%%%%%%%%%%%%%%%%%%%%%%%%%%%%%%%%%%%%%%%%%%%%%%%%%%%%%%%%%%%%%%%%%%%%%%%%%%%%%%%%%%%%%%%%%

\section{A Calabi type functional on the space of K\"ahler metrics} \label{sect:cal}

%%%%%%%%%%%%%%%%%%%%%%%%%%%%%%%%%%%%%%%%%%%%%%%%%%%%%%%%%%%%%%%%%%%%%%%%%%%%%%%%%%%%%%%%%%%%%%%%%%%%%%%%

We study the zeroes of the moment map $\mu$ on a closed K\"ahler manifold $(M,\omega,J)$.
Instead of looking for zeroes that are symplectic connections with respect to the fixed symplectic form $\omega$,
we will consider Levi-Civita connections associated to the K\"ahler forms $\tilde{\omega}$ in the same K\"ahler class than $\omega$.

%%%%%%%%%%%%%%%%%%%%%%%%%%%%%%%%%%%%%%%%%%%%%%%%%%%%%%%%%%%%%%%%%%%%%%%%%%%%%%%%%%%%%%%%%%%%%%%%%%%%%%%%

\subsection{The norm squared moment map functional}

%%%%%%%%%%%%%%%%%%%%%%%%%%%%%%%%%%%%%%%%%%%%%%%%%%%%%%%%%%%%%%%%%%%%%%%%%%%%%%%%%%%%%%%%%%%%%%%%%%%%%%%%

We consider a closed K\"ahler manifold $(M,\omega,J)$. Let $\Theta$ be the K\"ahler class of $\omega$ and denote
by $\MOm$ the set of K\"ahler forms in the class $\Theta$. By the classical $dd^c$-lemma,
$$\MOm:= \{\omega_{\phi}=\omega+dd^c\phi \textrm{ s.t. } \phi\in C^{\infty}_0(M),\ \omega_{\phi}(\cdot,J\cdot) \textrm{ is positive definite }\},$$
where $d^cF:=-dF\circ J$. The condition $\omega_{\phi}(\cdot,J\cdot)$ being positive definite is open. Then, $\MOm$ is a Fr\'echet manifold modeled on $C^{\infty}_0(M)$.

\begin{defi}
We define the {\bf $\F$-functional} to be the map
$$\F:\MOm \rightarrow \R:\omega_\phi \mapsto \F(\omega_\phi):= \int_M (\mu^{\phi}(\nabla^{\phi}))^2\dvolphi,$$
where $\nabla^{\phi}$ is the Levi-Civita connection of the K\"ahler metric $g_{\phi}(\cdot,\cdot):=\omega_{\phi}(\cdot,J\cdot)$ and $\mu^{\phi}$ is the normalised moment map on the space $\E(M,\omega_{\phi})$ described in Section \ref{sect:symplconn}.
\end{defi}

This new point of view of varying the K\"ahler form in a fixed K\"ahler class is intimately related to the one of 
varying the symplectic connection of the fixed symplectic form $\omega$. This relation is the key of the remainder of the paper.

Consider a smooth one-parameter family $\phi:\ ]-\epsilon,\epsilon[\rightarrow C^{\infty}_0(M):t\mapsto \phi(t)$ for some $\epsilon\in \R^+_0$ such that the $2$-form $\omega_{\phi(t)}:= \omega+ dd^c\phi(t)$ is a smooth path inside $\MOm$. All the forms $\omega_{\phi(t)}$ are symplectomorphic to each other. Indeed, set $X_t:=-\grad^{\phi(t)}(\dot{\phi})$ the gradient vector field of $\dot{\phi}(t)$ with respect to $g_{\phi(t)}$ (that is $g_{\phi(t)}(\grad^{\phi(t)}(\dot{\phi}),\cdot)=d\dot{\phi}$). Then the one parameter family of diffeomorphisms $f_t$ integrating the time-dependent vector field $X_t$ satisfies
\begin{equation} \label{eq:Moser}
f_t^*\omega_{\phi(t)}=\omega.
\end{equation}

Consider $f_t$ as in the above equation (\ref{eq:Moser}). Then, the natural action of $f_t^{-1}$ on $J$ produces a path
$$J_t:=f_t^{-1}.J:= f_{t*}^{-1} J f_{t*} \in \J_{int}(M,\omega).$$
Define the associated K\"ahler metric $g_{J_t}(\cdot,\cdot):=\omega(\cdot,J_t\cdot)$ and denote by $\nabla^{J_t}$ its
Levi-Civita connection. Then, $\nabla^{J_t}$ and $\nabla^{\phi(t)}$ are related by the following formula : 

\begin{prop} \label{prop:nablaJphi}
With the above notations, we have that $\nabla^{J_t}\in \E(M,\omega)$ and
$$\nabla^{J_t}=f_t^{-1}.\nabla^{\phi(t)},$$
where $(f_t^{-1}.\nabla^{\phi(t)})_YZ = f_{t*}^{-1}\nabla^{\phi(t)}_{f_{t*}Y}f_{t*}Z$.
\end{prop}

\begin{proof}
Because $(M,\omega,J_t)$ is K\"ahler, then $\nabla^{J_t}$ preserves $\omega$.
The equation $\nabla^{J_t}=f_t^{-1}.\nabla^{\phi(t)},$ follows from the observation that $g_{J_t}=f_t^*g_{\phi(t)}$.
\end{proof}

\begin{lemme} \label{lemme:main}
Let $\phi:\ ]-\epsilon,\epsilon[\rightarrow C^{\infty}_0(M):t\mapsto \phi(t)$ for some $\epsilon\in \R^+_0$ be a smooth map such that the $2$-form $\omega_{\phi(t)}:= \omega+ dd^c\phi(t)$ belongs to $\MOm$. If $f_t$ is the family of diffeomorphisms defined by equation (\ref{eq:Moser}). Denote by $\nabla^{\phi(t)}$ (resp. $\nabla^{J_t}$) the Levi-Civita connection of the K\"ahler metric $g_{\phi(t)}$ (resp. $g_{J_t}$). Then,
$$\mu(\nabla^{J_t})=f_t^*\mu^{\phi(t)}(\nabla^{\phi(t)}).$$
\end{lemme}

\begin{proof}
By Proposition \ref{prop:nablaJphi}, we have $\nabla^{J_t}=f_t^{-1}.\nabla^{\phi(t)}$. In terms
of the curvature tensors, it means that 
$$R^{\nabla^{J_t}}=f_t^*R^{\nabla^{\phi(t)}}.$$
Since $f_t^*\omega_{\phi(t)}=\omega$, a direct computation leads to 
$$\tilde{\mu}(\nabla^{J_t})=f_t^*\tilde{\mu}^{\phi(t)}(\nabla^{\phi(t)}).$$
The above equation implies that the integral
$$\int_M \tilde{\mu}^{\phi(t)}(\nabla^{\phi(t)})\frac{\omega_{\phi(t)}^n}{n!}$$
does not depend on $\phi(t)$. So that the normalised moment maps satisfy
$$\mu(\nabla^{J_t})=f_t^*\mu^{\phi(t)}(\nabla^{\phi(t)}).$$
\end{proof}

%%%%%%%%%%%%%%%%%%%%%%%%%%%%%%%%%%%%%%%%%%%%%%%%%%%%%%%%%%%%%%%%%%%%%%%%%%%%%%%%%%%%%%%%%%%%%%%%%%%%%%%%

\subsection{An associated elliptic operator}

%%%%%%%%%%%%%%%%%%%%%%%%%%%%%%%%%%%%%%%%%%%%%%%%%%%%%%%%%%%%%%%%%%%%%%%%%%%%%%%%%%%%%%%%%%%%%%%%%%%%%%%%

We show the problem of finding zeroes of $\mu^{\phi}(\nabla^{\phi})$ on $\MOm$ is an elliptic partial differential problem.
We mean that its linearization is given by an elliptic differential operator with smooth coefficients and with leading term the operator $\Delta^3$, the cube of the Laplacian.

\begin{prop} \label{prop:mudpdphi}
The moment map $\mu^{\phi}(\nabla^{\phi})$ depends analytically on the K\"ahler potential $\phi$ and its derivatives up to order $6$.
\end{prop}

\begin{proof}
 Recall that
$$ \mu^{\phi}(\nabla^{\phi})= -\frac{1}{2}\Delta^{\phi} Scal^{\nabla^{\phi}}  -\frac{1}{2} Ric^{\nabla^{\phi}}_{pq}(Ric^{\nabla^{\phi}})^{pq} + \frac{1}{4} R^{\nabla^{\phi}}_{pqrs}(R^{\nabla^{\phi}})^{pqrs}-\mu_0.$$
The curvature tensor $R^{\nabla^{\phi}}$ depends analytically on the K\"ahler potential $\phi$ and its derivatives up to order four. Indices are raised using the inverse of the K\"ahler form which depends analytically on $\phi$ and its derivatives up to order $2$. The leading term $-\frac{1}{2}\Delta^{\phi} Scal^{\nabla^{\phi}}$ involves derivatives of the K\"ahler potential up to order $6$.
\end{proof}

\begin{defi} \label{def:Dphi}
Let $\omega+dd^c\phi \in \MOm$. We define the {\bf operator $D^{\phi}:C^{\infty}(M)\rightarrow C^{\infty}(M)$} by : 
\begin{equation}
D^{\phi}(\psi):= \frac{d}{dt}|_0 \mu^{\phi+t\psi}(\nabla^{\phi+t\psi})
\end{equation}
for any $\psi \in C^{\infty}(M)$
\end{defi}

\begin{prop} \label{prop:Dphi}
\begin{enumerate}
\item[]
\item The operator $D^{\phi}$ is an elliptic partial differential operator of order $6$ null on constants with smooth coefficients depending analytically on the derivatives of $\phi$ up to order $6$.
\item If $(M,\omega_{\phi},J)$ satisfies the condition (C) and $\mu^{\phi}(\nabla^{\phi})=0$, then $\psi\in C^{\infty}(M)$ is in the kernel of $D^{\phi}$ if and only if $\Lr_{X^{\omega_{\phi}}_{\psi}}J=0$, where $X^{\omega_{\phi}}_{\psi}$ denotes the Hamiltonian vector field of $\psi$ with respect to $\omega_{\phi}$.
\end{enumerate}
\end{prop}

\begin{proof}
\begin{enumerate}
\item[]
\item From Proposition \ref{prop:mudpdphi}, we deduce the analytic dependence of the coefficient of $D^{\phi}$ on $\phi$ and its derivatives up to order $6$. The leading term of $D^{\phi}$ occurs in $\frac{d}{dt}|_0\Delta^{\phi+t\psi} Scal^{\nabla^{\phi+t\psi}}$. Recall that
\begin{eqnarray*}
\frac{d}{dt}|_0 Scal^{\nabla^{\phi+t\psi}} & = & -(\Delta^{\phi})^2\psi -  (dd^c\psi)_{pq} (\rho^{\nabla^{\phi}})^{pq}, \\
\frac{d}{dt}|_0\Delta^{\phi+t\psi} F & = & (dd^cF)_{pq}(dd^c\psi)^{pq}.
\end{eqnarray*}
where $\rho^{\nabla^{\phi}}(.,.):=Ric^{\nabla^{\phi}}(J.,.)$ is the Ricci form and indices are rised using the symplectic form $\omega_{\phi}$.
Then, we have
$$D^{\phi}\psi=\frac{1}{2}(\Delta^{\phi})^3 \psi + \textrm{lower order terms} .$$
It means $D^{\phi}$ is an elliptic operator.
\item We assume $\phi=0$ for simplicity. By hypothesis we know $(M,\omega,J)$ satisfies the condition (C) and $\mu(\nabla)=0$. We compute 
$$\int_M D\psi_1 \psi_2\dvol= \frac{d}{dt}|_0\int_M \mu^{t\psi_1}(\nabla^{t\psi_1})\psi_2 \dvol$$
 for $\psi_1,\psi_2 \in C^{\infty}(M)$ (where $D$ stands for $D^{\phi}$ with $\phi=0$). Set $f_t$ the family of diffeomorphisms of $M$ generated by $X_t:= -grad^{t\psi_1}(\psi_1)$ so that $f_t^*\omega_{t\psi_1}=\omega$ for small $t$. Then, by Lemma \ref{lemme:main},
$$\int_M D\psi_1 \psi_2\dvol= \frac{d}{dt}|_0\int_M (f_t^{-1})^*\mu(\nabla^{J_t})\psi_2 \dvol,$$
for the family of complex structures $J_t:=(f_t^{-1}).J$. Using the moment map property of $\mu$ and the fact that $\mu(\nabla)=0$, we obtain :
$$\frac{d}{dt}|_0\int_M (f_t^{-1})^*\mu(\nabla^{J_t})\psi_2 \dvol = \Omega^{\E}_{\nabla}(\Lr_{X_{\psi_2}}\nabla, \frac{d}{dt}|_0\nabla^{J_t}).$$
Now, $\nabla^{J_t}=\LC(J_t)$ and $J_t=f_t^{-1}.J:=f_{t*}^{-1}Jf_{t*}$. So that, $\frac{d}{dt}J_t=\Lr_{f_{t*}^{-1}X_t}J_t$ and, by direct computation, one checks that $f_{t*}^{-1}X_t=-J_tX_{f_t^*\psi_1}$. So that $\frac{d}{dt}J_t= -\Lr_{J_tX_{f_t^*\psi_1}}J_t=-J_t\Lr_{X_{f_t^*\psi_1}}J_t$. Then, we have 
$$\Omega^{\E}_{\nabla}(\Lr_{X_{\psi_2}}\nabla, \frac{d}{dt}|_0\nabla^{J_t}) = -\Omega^{\E}_{\nabla}\left(\LC_{*J}(\Lr_{X_{\psi_2}}J),\LC_{*J}(J\Lr_{X_{\psi_1}}J)\right).$$
Because $(M,\omega,J)$ satisfies the condition C, we see that $\int_M D\psi_1 \psi_2\dvol=0$ for all $\psi_2$ if and only if $\Lr_{X_{\psi_1}}J=0$.
\end{enumerate}
\end{proof}

\begin{ex}
Consider the flat torus $(\C^{n}/\Z^{2n}, \omega_{std},i)$ with its flat connection $\nabla$. From the computations of point $1$ in the above proof, at $\omega_{\phi}=\omega_{std}$, we have 
$$D^{\phi}\psi=\frac{1}{2}(\Delta^{\phi})^3\psi.$$
\end{ex}

\begin{defi}\label{def:Dphi*}
The {\bf operator $(D^{\phi})^*$} is the formal adjoint of the operator $D^{\phi}$ with respect to the K\"ahler form $\omega_{\phi}$. That is $(D^{\phi})^*$ is defined by the equation
$$\int_M FD^{\phi}G \dvolphi=\int_M (D^{\phi})^*F.G \dvolphi.$$
\end{defi}

\noindent Because $(D^{\phi})^*$ is the formal adjoint of the operator $D^{\phi}$, the following Proposition is obvious.

\begin{prop} \label{prop:Dphi*}
The operator $(D^{\phi})^*$ is an elliptic partial differential operators of order $6$ with smooth coefficients depending analytically on the derivatives of $\phi$ up to order $12$.
\end{prop}

%%%%%%%%%%%%%%%%%%%%%%%%%%%%%%%%%%%%%%%%%%%%%%%%%%%%%%%%%%%%%%%%%%%%%%%%%%%%%%%%%%%%%%%%%%%%%%%%%%%%%%%%

\subsection{Manifold structure for $\F^{-1}(0)$}

%%%%%%%%%%%%%%%%%%%%%%%%%%%%%%%%%%%%%%%%%%%%%%%%%%%%%%%%%%%%%%%%%%%%%%%%%%%%%%%%%%%%%%%%%%%%%%%%%%%%%%%%

We analyse the structure of the space of zeroes of $\F$ or equivalently the space of $\omega_{\phi}\in \MOm$ such that $\mu^{\phi}(\nabla^{\phi})=0$, when $(M,\omega_{\phi},J)$ satisfies the condition (C). The key property is that the Hessian of $\F$ at a zero is given by a non-negative elliptic operator.

\begin{prop} \label{prop:dF}
The map $\F:\MOm\rightarrow \R$ is smooth and its differential at $\omega_{\phi}$ evaluated at $\psi \in T_{\omega_{\phi}}\MOm \simeq C^{\infty}_0(M)$ is 
$$d\F_{\phi}(\psi):=\int_M \left(2(D^{\phi})^*\mu^{\phi}(\nabla^{\phi})-\Delta^{\phi}(\mu^{\phi}(\nabla^{\phi}))^2\right)\psi \dvolphi.$$
Moreover, if $\omega_{\phi}$ is a zero of $\F$, then it is a critical point of $\F$ and the Hessian $d^2\F_{\phi}$ of $\F$ at $\omega_{\phi}$ is given by
$$d^2\F_{\phi}(\psi_1,\psi_2):=\int_M \left(2(D^{\phi})^*D^{\phi}\psi_1\right)\psi_2 \dvolphi,$$
for $\psi_1,\psi_2 \in C^{\infty}_0(M)$
\end{prop}

\begin{proof}
The fact that $\F$ is smooth directly follows from Proposition \ref{prop:Dphi}.

\noindent Now, we compute the differential of $\F$ :
\begin{eqnarray*}
\frac{d}{dt}|_0 \F(\omega_{\phi+t\psi}) & = & 2 \int_M \left(\frac{d}{dt}|_0\mu^{\phi+t\psi}(\nabla^{\phi+t\psi})\right)\mu^{\phi}(\nabla^{\phi})  \dvolphi + \int_M \mu^{\phi}(\nabla^{\phi})^2 \frac{d}{dt}|_0 \frac{\omega_{\phi+t\psi}^n}{n!} \\
 & = & 2 \int_M \left(D^{\phi} \psi \right)\mu^{\phi}(\nabla^{\phi})  \dvolphi - \int_M \mu^{\phi}(\nabla^{\phi})^2 \Delta^{\phi} \psi \dvolphi, \\
 & = & \int_M \left(2(D^{\phi})^*\mu^{\phi}(\nabla^{\phi})-\Delta^{\phi}(\mu^{\phi}(\nabla^{\phi}))^2\right)\psi \dvolphi.
\end{eqnarray*}

\noindent Because it is symmetric, the Hessian of $\F$ is determined by $d^2\F_{\phi}(\psi,\psi)$ for $\psi \in C^{\infty}_0(M)$. We compute
\begin{eqnarray*}
d^2\F_{\phi}(\psi,\psi) & = & \frac{d^2}{dt^2}|_0 \F(\omega_{\phi+t\psi}) \\
 & = & \frac{d}{dt}|_0 \int_M \left(2(D^{\phi+t\psi})^*\mu^{\phi+t\psi}(\nabla^{\phi+t\psi})-\Delta^{\phi+t\psi}(\mu^{\phi+t\psi}(\nabla^{\phi+t\psi}))^2\right)\psi \frac{\omega^n_{\phi+t\psi}}{n!}.
\end{eqnarray*}
Since $\mu^{\phi}(\nabla^{\phi})=0$, we have
$$d^2\F_{\phi}(\psi,\psi) =   \int_M \left(2(D^{\phi})^*\frac{d}{dt}|_0\mu^{\phi+t\psi}(\nabla^{\phi+t\psi})\right)\psi \dvolphi = \int_M \left(2(D^{\phi})^*D^{\phi}\psi\right)\psi \dvolphi.$$
\end{proof}

From now on, we will work in a neighbourhood of a given $\omega\in \MOm$. Let $U\in C_0^{\infty}(M)$ be a convex neighbourhood of the origin such that if $\phi\in U$ then $\omega_{\phi}\in \MOm$. Then $d\F$ induces a smooth map of Fr\'echet spaces
$$\widetilde{d\F}:U\rightarrow C_0^{\infty}(M)$$
defined for $\phi\in U$ by 
$$\widetilde{d\F}(\phi):=2(D^{\phi})^*\mu^{\phi}(\nabla^{\phi})-\Delta^{\phi}(\mu^{\phi}(\nabla^{\phi}))^2-\int_M 2(D^{\phi})^*\mu^{\phi}(\nabla^{\phi})-\Delta^{\phi}(\mu^{\phi}(\nabla^{\phi}))^2 \dvol.$$
Moreover, from Proposition \ref{prop:dF}, $\omega_{\phi}$ with $\phi \in U$ is a critical point of $\F$ if and only if $\widetilde{d\F}(\phi)\equiv 0$.

We will now extend the map $\widetilde{d\F}$ to a smooth map defined on suitable Sobolev spaces. Denote by $\|.\|_0$ the $L^2$-norm induced by the K\"ahler metric $g$ on $C^{\infty}(M)$. Let $l>0$ be an integer, for $K\in C^{\infty}(M)$, $\nabla^l K$ is a section of $T^*M^{\otimes l}$. The K\"ahler metric $g$ induces a scalar product and then a $L^2$-norm $\|\cdot\|_{g,l}$ on $\Gamma(T^*M^{\otimes l})$ we define the $l$-th Sobolev norm $\|.\|_l$ on $C^{\infty}(M)$ by : 
$$\|H\|_l:= \left(\|H\|_{0}^2+\sum_{j=1}^{l} \| \nabla^j H\|_{g,j}^2\right)^{\frac{1}{2}}.$$
The $l$-th Sobolev space $H^l(M)$ is defined to be the completion of $C^{\infty}(M)$ for the $l$-th Sobolev norm. The $H^l(M)$ are Hilbert spaces. The subspaces $H^l_0(M)$ denote the closure of $C^{\infty}_0(M)$ in $H^l(M)$, they are also  Hilbert spaces.

\begin{prop}
For $l\geq \frac{\dim(M)}{2}+12$, the map $\widetilde{d\F}$ extends to a smooth map 
$$\widetilde{d\F}:\widetilde{U}\subset H^l_0(M) \rightarrow H^{l-12}_0(M),$$
for $\widetilde{U}$ a neighbourhood of the origin in $H^l_0(M)$.
\end{prop}

\begin{proof}
When $l\geq \frac{\dim(M)}{2}+12$, the Sobolev embedding's Theorem states the inclusions $H^l_0(M)\hookrightarrow C^{12}(M)$ and $H^{l-12}_0(M)\hookrightarrow C^{0}(M)$ are continuous, see for example \cite{Aub}. Since $\widetilde{d\F}$ is continuous for the Fr\'echet topology on $U$, it extends to a continuous map defined on a neighbourhood $V$ of the origin in $C^{12}(M)$ with value in $C^0(M)$. By Propositions \ref{prop:Dphi} and \ref{prop:Dphi*}, $\widetilde{d\F}(\phi)$ depends analytically on $\phi$ and its derivatives of order at most $12$ so that if $\phi \in V \cap H^{l}_0(M)$ then $\widetilde{d\F}(\phi)$ sits in $H^{l-12}_0(M)$. Then, the restriction $\widetilde{d\F}:V\cap H^{l}_0(M) \rightarrow H^{l-12}_0(M)$ is a smooth map of Hilbert manifolds.
\end{proof}

To finish this subsection, we will assume $\omega$ is a zero of $\F$ satisfying the condition C and we show that the zeroes of $\F$ around $\omega$ form a manifold. We will write $D$ for $D^{\phi}$ when $\phi=0$. 

Let $(M,\omega,J)$ be a K\"ahler manifold and consider its reduced automorphism group $H_0(M,J)$ that is the connected (finite dimensional) Lie group whose Lie algebra $\mathfrak{h}_0$ consists of real holomorphic vector fields $Z$ (i.e. $\Lr_ZJ=0$) of the form $X_H+JX_K$ for $H$ and $K\in C^{\infty}_0(M)$. 

If $\omega$ is a zero of $\F$, then Lemma \ref{lemme:main} with a family $f_t\in H_0(M,J)$ shows that all the points in the $H_0(M,J)$-orbit $\mathcal{O}$ of $\omega$ are zeroes of $\F$. The tangent space to $\mathcal{O}$ is the space of $K\in C^{\infty}_0(M)$ such that there exists $H \in C^{\infty}_0(M)$ with $X_H+JX_K\in \mathfrak{h}_0$. Moreover, if $\omega$ satisfies the Condition C, then, by Proposition \ref{prop:condCcplx}, all points in $\mathcal{O}$ satisfy this condition. Our main Theorem is :
\begin{theoremprinc}
Assume the K\"ahler manifold $(M,\omega,J)$ has a Ricci tensor which is non-negative everywhere. 

If $\omega$ is a zero of $\F$, then there exists an open neighbourhood $U$ of $\omega$ in $\MOm$ such that the only zeroes of $\F$ inside $U$ is the $H_0(M,J)$-orbit of $\omega$, where $H_0(M,J)$ is the reduced automorphism group of $(M,\omega,J)$.
\end{theoremprinc}

\noindent The proof of Theorem \ref{theorprinc:zeroF} is a direct corollary of the following local statement.
 
\begin{theorem} \label{theor:loctheor1}
Let $(M,\omega,J)$ be a K\"ahler manifold satisfying the Condition C and such that $\omega$ is a zero of $\F$. Then, $\widetilde{d\F}(0)=0$ and $(\widetilde{d\F})^{-1}(0)$ is a submanifold of $\widetilde{U}$ (shrinking $\widetilde{U}$ if necessary) which is the intersection of $U$ with the $H_0(M,J)$-orbit of $\omega$ .
\end{theorem}

\begin{proof}
We use the implicit function theorem for Hilbert manifolds. When $\omega$ is a zero of $\F$, it is also a zero of $\mu$, so that $\widetilde{d\F}(0)=0$. 

Let $l\geq \frac{\dim(M)}{2}+12$ be an integer. From Proposition \ref{prop:dF}, the differential at $0$ of $\widetilde{d\F}$ evaluated at $\psi\in H^{l}_0(M)$ is given, by 
\begin{equation}
d_0\widetilde{d\F}(\psi) = 2D^*D\psi - \int_M 2D^*D\psi \dvol =  2D^*D\psi,
\end{equation}
where $D^*D$ is extended by continuity to an operator $H^l_0(M)\rightarrow H^{l-12}_0(M)$ still denoted by $D^*D$. 

By Proposition \ref{prop:Dphi}, the self-adjoint operator $D^*D$ is elliptic with smooth coefficients. So that, $\ker(D^*D)\subset C^{\infty}_0(M)$ is finite dimensional and for $s\in \Z$ we have the isomorphism of Hilbert space :
$$ H^{s}_0(M)\cong D^*D(H^{s+12}_0(M))\oplus \ker(D^*D),$$
the summands are orthogonal for the $L^2$-product and the projections are continuous. So that the differential of $\widetilde{d\F}$ has finite dimensional kernel $\ker(D^*D)$ and image equal to $D^*D(H_0^l(M))$. This is enough to use the implicit function theorem on Hilbert spaces and concludes that $\widetilde{d\F}^{-1}(\ker(D^*D))$ is a submanifold of $\widetilde{U}$ whose tangent space is isomorphic to $\ker(D^*D)$.

To conclude the proof, just consider $\mathcal{O}$ the $H_0(M,J)$-orbit of $\omega$ and the intersection $\mathcal{O}\cap U$ seen as a subset of $\widetilde{U}$. Any point in $\mathcal{O}\cap U$ satisfy the condition (C). By Proposition \ref{prop:Dphi}, the tangent space of $\mathcal{O}\cap U$ at some point $\phi$ contains $\ker((D^{\phi})^*D^{\phi})\cong \ker(D^*D)$. Moreover, $\widetilde{d\F}$ is constant on $\mathcal{O}\cap U$. So that $\widetilde{d\F}^{-1}(\ker(D^*D))=\widetilde{d\F}^{-1}(0)=\mathcal{O}\cap U$. The proof is over.
\end{proof}

\begin{proof}[Proof of Theorem \ref{theorprinc:zeroF}]
Since a K\"ahler manifold with non-negative Ricci tensor everywhere satisfies the Condition C, we can use the above Theorem \ref{theor:loctheor1}. So that, there exists a neighbourhood $U$ of $\omega$ in $\MOm$ such that $(\widetilde{d\F})^{-1}(0)=U\cap \mathcal{O}$. Then, the zeroes of $\F$ restricted to $U$ are elements of $\mathcal{O}$.
\end{proof}

%%%%%%%%%%%%%%%%%%%%%%%%%%%%%%%%%%%%%%%%%%%%%%%%%%%%%%%%%%%%%%%%%%%%%%%%%%%%%%%%%%%%%%%%%%%%%%%%%%%%%%%%

\subsection{Critical points}

%%%%%%%%%%%%%%%%%%%%%%%%%%%%%%%%%%%%%%%%%%%%%%%%%%%%%%%%%%%%%%%%%%%%%%%%%%%%%%%%%%%%%%%%%%%%%%%%%%%%%%%%

We compute $d\F$ from a different point of view as in Proposition \ref{prop:dF} to obtain an equation for critical points of $\F$ similar to the one of extremal K\"ahler metric. This equation was already pointed out in Fox's paper \cite{Fox} but in different settings.

\begin{theorem} \label{theor:critF}
Let $(M,\omega,J)$ be a K\"ahler manifold satisfying the Condition C. Then, the form $\omega$ is a critical point of $\F$ if and only if $\Lr_{X_{\mu(\nabla)}}\nabla=0$, that is, if and only if the Hamiltonian vector field $X_{\mu(\nabla)}$ is a Killing vector field.
\end{theorem}

\noindent Let us compute the first variation of $\F(\omega_{\phi})$.

\begin{prop} \label{prop:firstvarF}
Let $\omega_{\phi(t)}$ be a smooth path in $\MOm$ with $\phi(0)=0$.
Then, 
$$\frac{d}{dt} \F(\omega_{\phi(t)})= - 2 \Omega^{\E}\left(\LC_{*J_t}(\Lr_{X_{\mu(\nabla^{J_t})}}J_t),\LC_{*J_t}(J_t\Lr_{X^{\omega}_{f_t^*\dot{\phi}}}J_t)\right).$$
\end{prop}

\begin{proof}
Let $f_t \in \Diff_0(M)$ defined as in equation (\ref{eq:Moser}). With the help of the Lemma \ref{lemme:main}, we compute : 
\begin{eqnarray*}
\frac{d}{dt} \F(\omega_{\phi(t)}) & = & \frac{d}{dt} \int_M (\mu(\nabla^{\phi(t)}))^2\frac{\omega_{\phi(t)}^n}{n!},\\
 & = & \frac{d}{dt} \int_M (f_t^*\mu(\nabla^{\phi(t)}))^2\dvol, \\
 & = & \frac{d}{dt} \int_M (\mu(\nabla^{J_t})))^2\dvol, \\
 & = &  2\Omega^{\E}\left(\Lr_{X_{\mu(\nabla^{J_t})}}\nabla^{J_t},\frac{d}{dt}\nabla^{J_t}\right).
\end{eqnarray*}
where the last equality follows from the fact that $\mu$ is a moment map on $\E(M,\omega)$.

Now, $\nabla^{J_t}=\LC(J_t)$ and $J_t=f_t^{-1}.J:=f_{t*}^{-1}Jf_{t*}$. So that, $\frac{d}{dt}J_t=\Lr_{f_{t*}^{-1}X_t}J_t$ and $f_{t*}^{-1}X_t=-J_tX^{\omega}_{f_t^*\dot{\phi}}$. Then, 
$$\frac{d}{dt} \F(\omega_{\phi(t)})= -2 \Omega^{\E}\left(\LC_{^*J_t}(\Lr_{X_{\mu(\nabla^{J_t})}}J_t),\LC_{^*J_t}(J_t\Lr_{X_{f_t^*\dot{\phi}}}J_t)\right).$$
\end{proof}

\begin{proof}[Proof of Theorem \ref{theor:critF}]
We assume $(M,\omega,J)$ satisfies Condition C. The form $\omega$ is a critical point if and only if for all $H\in C^{\infty}_0(M)$ we have $\frac{d}{dt}|_0\F(\omega_{tH})=0$. By proposition \ref{prop:firstvarF}, this means 
$$\Omega^{\E}\left(\LC_{^*J}(\Lr_{X_{\mu(\nabla)}}J),\LC_{^*J}(J\Lr_{X_{H}}J)\right)=0$$
for all $H\in C^{\infty}(M)$. Now, Condition C implies that $\Lr_{X_{\mu(\nabla)}}J=0$.
\end{proof}

\begin{cor}
Let $(M,\omega,J)$ be a closed K\"ahler manifold which is Ricci flat. Then $\omega$ is a critical point of $\F$ if and only if it is a zero of $\F$.
\end{cor}

\begin{proof}
A Ricci flat K\"ahler manifold satisfies the Condition C. If $\omega$ is a critical point of $\F$, then $\Lr_{X_{\mu(\nabla)}}J=0$, by Theorem \ref{theor:critF}. On a Ricci flat manifold $\Lr_{X_{\mu(\nabla)}}J=0$ implies $\mu(\nabla)=0$ (the function $\mu$ being normalized), so that $\F(\omega)=0$.
\end{proof}

\section{Deformation quantization}

We will now provide a new motivation for the study of various moment maps on infinite dimensional symplectic manifolds. We exhibit a link between the moment map $\mu$ on the space of symplectic connections and the trace density of the Fedosov's star products. 

Moreover, we also mention other examples of star products, for which the trace density is linked to moment maps on infinite dimensional symplectic manifolds.

\subsection{Closed Fedosov's star products}

The moment map $\mu$ and its zeroes do have a nice interpretation in terms of Fedosov's star products.

Consider the space $C^{\infty}(M)[[\nu]]$ of formal power series of smooth functions. 
A star product \cite{BFFLS} on a symplectic manifold is a $\R[[\nu]]$-bilinear map
$$*:C^{\infty}(M)[[\nu]]\times C^{\infty}(M)[[\nu]] \rightarrow C^{\infty}(M)[[\nu]]:(F,H)\mapsto F*H=\sum_{r=0} \nu^r C_r(F,H),$$
such that : 
\begin{itemize}
\item $*$ is associative,
\item the $C_r$ are $\R[[\nu]]$-linear bidifferential operators,
\item $C_0(F,H)=FH$ and $C_1^-(F,H):= C_1(F,H)-C_1(H,F)=\{F,H\}$ where $\{F,H\}:=-\omega(X_F,X_H)$, for $F, H \in C^{\infty}(M)$,
\item $F*1=F=1*F$, for all $F \in C^{\infty}(M)$.
\end{itemize}

In \cite{fed2}, Fedosov gave a geometric construction of star products $*_{\nabla,\Omega}$ on symplectic manifold using a symplectic connection $\nabla$ and a formal series of closed $2$-forms $\Omega\in \nu \Omega^2(M)[[\nu]]$. In this subsection, we only consider Fedosov's star products obtained with $\Omega=0$. Concretely, the first terms up to order $3$ in $\nu$ are described by the formula
\begin{equation}
F*_{\nabla,0}H=FH+\frac{\nu}{2} \{F,H\}+ \frac{\nu^2}{4} \Lambda^{i_1j_1}\Lambda^{i_2j_2}(\nabla^2F)_{i_1i_2}(\nabla^2H)_{j_1j_2} + \frac{\nu^3}{48}S^3_{\nabla}(F,H) + O(\nu^4),
\end{equation}
where $F,H \in C^{\infty}(M)$ and, denoting by $\underline{\Lr_{X_F}\nabla}$ the symmetric $3$-tensor $\omega(\Lr_{X_F}\nabla(\cdot)\cdot,\cdot)$, 
$$S^3_{\nabla}(F,H):=\Lambda^{i_1j_1}\Lambda^{i_2j_2}\Lambda^{i_3j_3}(\underline{\Lr_{X_F}\nabla})_{i_1i_2i_3}(\underline{\Lr_{X_H}\nabla})_{j_1j_2j_3},$$
One checks that the $*_{\nabla,0}$-commutator is given by 
$$[F,H]_{*_{\nabla,0}}=F*_{\nabla,0}H-H*_{\nabla,0}F=\nu \{F,H\}+\frac{\nu^3}{24}S^3_{\nabla}(F,H) + O(\nu^4),$$

\begin{rem}
In \cite{gr3} the notion of ``natural'' star products was introduced. Such natural star products determine the symplectic connection. Star products obtained via the Fedosov's method are natural in the sense of \cite{gr3}.
\end{rem}

Let $*$ be a star product on a symplectic manifold.
A {\bf trace} for $*$ is a $\R[[\nu]]$-linear map
$$\tr : C^{\infty}_{c}(M)[[\nu]] \rightarrow \R[[\nu]],$$
satisfying $\tr(F*H)=\tr(H*F)$ for all $F,H \in C^{\infty}_{c}(M)[[\nu]]$, where $C^{\infty}_{c}(M)$ denotes the space of smooth functions with compact support.

\begin{prop}[Fedosov \cite{fed3,fed}, Nest-Tsigan \cite{NT}, Gutt-Rawnsley \cite{gr}] \label{prop:existtrace} 
Any star product $*$ on a symplectic manifold $(M,\omega)$ admit a trace. More precisely, 
there exists $\rho \in C^{\infty}(M)[[\nu]]$ such that
\begin{equation}
\tr(F):= \int_M F\rho \dvol
\end{equation}
for all $F \in C^{\infty}_{c}(M)[[\nu]]$. The function $\rho$ is called a trace density.

Moreover, any two traces for $*$ differ from each other by multiplication with a formal constant $C \in \R[\nu^{-1},\nu]]$.
\end{prop}

A star product is called closed up to order $l$ in $\nu$ if the map $F\mapsto \int_M F \dvol$ satisfies the trace property
up to order $l$ in $\nu$, i.e.
\begin{equation}
\int_M F*H \dvol = \int_M H*F \dvol + O(\nu^{l+1}), \  \textrm{ for all } F,H \in C^{\infty}_{c}(M)[[\nu]].
\end{equation}

Fedosov \cite{fed4} gives an algorithm to compute the trace density of the star product $*_{\nabla,0}$. He computes explicitely the first non trivial term. Here we prove that the equivariant moment map property of $\mu$ implies that $\mu(\nabla)$ is the first non trivial term of a trace density for $*_{\nabla,0}$. 

\begin{prop}[Fedosov \cite{fed4}] \label{prop:trace}
Let $(M,\omega)$ be a closed symplectic manifold. If $F,H \in C^{\infty}(M)$, then $\rho:= 1+\frac{\nu^2}{24}\mu(\nabla)$satisfies 
$$\int_M (F*_{\nabla,0}H-H*_{\nabla,0}F)\rho \dvol \equiv 0 \textrm{ mod } O(\nu^4).$$

Consequently, $*_{\nabla,0}$ is closed up to order $3$ in $\nu$ if and only if $\nabla$ is a zero of the normalised moment map $\mu$.
\end{prop}

\begin{proof}
We compute 
$$\int_M [F,H]_{*_{\nabla,0}}\rho \dvol=\frac{\nu^3}{24}\left(\Omega^{\E}(\Lr_{X_F}\nabla,\Lr_{X_H}\nabla)+\int_M\{F,H\}\mu(\nabla)\dvol\right)+ O(\nu^4).$$
Because $\mu$ is an equivariant moment map on $\E(M,\omega)$, equation (\ref{eq:momentmu}) says : 
$$\Omega^{\E}(\Lr_{X_H}\nabla,\Lr_{X_F}\nabla)= - \int_M\{F,H\}\mu(\nabla)\dvol.$$
It concludes the proof.
\end{proof}

Let $(M,\omega,J)$ be a closed K\"ahler manifold. One associate naturally the Fedosov's star product $*_{\nabla,0}$ for $\nabla$ the Levi-Civita connection. Assume $\widetilde{\omega}\in \MOm$, with $\Theta:=[\omega]$, is such that $f^*\widetilde{\omega}=\omega$ for an $f\in H_0(M,J)$. Set $\widetilde{\nabla}$ the Levi-Civita connection of $\widetilde{\omega}$. One checks $\nabla=f^{-1}.\widetilde{\nabla}$. Because of that, see \cite{fed2}, the pull-back by $f^*$ gives an isomorphism of star product 
$$f^*:(C^{\infty}(M)[[\nu]],*_{\widetilde{\nabla},0})\stackrel{\cong}{\rightarrow}(C^{\infty}(M)[[\nu]],*_{\nabla,0}).$$
Our Theorem \ref{theorprinc:zeroF} translates in terms of Fedosov's star product in the following way.

\begin{theoremprinc} \label{theor:closed*Fed}
Let $(M,\omega,J)$ be a closed K\"ahler manifold with Levi-Civita connection $\nabla$ and Ricci tensor everywhere non negative. Assume the Fedosov's star product $*_{\nabla,0}$ is closed up to order $3$ in $\nu$.

Then, there exists an open neighbourhood $U$ of $\omega$ in $\MOm$ such that if $\widetilde{\omega}\in U$ gives rise
to the star product $*_{\widetilde{\nabla},0}$ closed up to order $3$ in $\nu$ then there is an $f\in H_0(M,J)$ inducing an isomorphism
$$f^*:(C^{\infty}(M)[[\nu]],*_{\widetilde{\nabla},0})\stackrel{\cong}{\rightarrow}(C^{\infty}(M)[[\nu]],*_{\nabla,0}).$$
\end{theoremprinc}

\begin{proof}
Because $*_{\nabla,0}$ is closed up to order $3$, $\F(\omega)=0$. Now, we can use Theorem \ref{theorprinc:zeroF} which states that in a neighbourhood of $\omega$ in $\MOm$ the zeroes of $\F$ are the $H_0(M,J)$-orbit of $\omega$. As pointed out above, an element $f\in H_0(M,J)$ produces the desired isomorphism of star product algebra.
\end{proof}

\begin{ex}
Consider the flat torus $(\C^{n}/\Z^{2n}, \omega_{std},i)$ with its flat connection $\nabla$. The corresponding Fedosov's star product $*_{\nabla,0}$ is closed (i.e. the integral is a trace functional). Since $Ric^{\nabla}\equiv 0$, the group $H_0(M,J)=0$. By the above Theorem \ref{theor:closed*Fed}, $\omega_{std}$ is isolated from any other $\omega_{\phi} \in \mathcal{M}_{[\omega_{std}]}$ such that $*_{\nabla^{\phi},0}$ is closed. 
\end{ex}

\begin{ex}
Consider $(\C P^n, \omega_{FS},J)$ the complex projective space with its Fubini-Study metric $\omega_{FS}$ and standard $J$. Denote by $\nabla$ the K\"ahler connection, it is known that $*_{\nabla,0}$ is closed. Because $Ric^{\nabla}$ is positive definite at any point, Theorem \ref{theor:closed*Fed} states that in a neighbourhood $U$ of $\omega_{FS}$ in $\mathcal{M}_{[\omega_{FS}]}$ the only other closed Fedosov's star products $*_{\widetilde{\nabla},0}$ associated to $\widetilde{\omega}\in U$ are obtained by the pull-back by a transformation of $SL(n+1,\C)$.
\end{ex}

\subsection{Other closed star products and moment maps} \label{subsect:other*}

The link between trace densities for star products and moment maps on infinite dimensional symplectic manifolds in Proposition \ref{prop:trace} involves more star products than just the Fedosov's star products $*_{\nabla, 0}$. We give here two additional examples.

\subsubsection{Wick star products of Bordemann-Waldmann} \label{subsubsect:wick}

On a K\"ahler manifold, Bordemann and Waldmann \cite{BW} adapted the Fedosov's construction to produce star products of Wick type. That is star products defined by series of bidifferential operators $C_j$ for which the first argument is differentiated in holomorphic directions while the other is differentiated in anti-holomorphic directions (for their construction, the star product is defined on $C^{\infty}(M,\C)[[\nu]]$ and the $C_j$'s are $\C[[\nu]]$-bilinear).

\begin{theorem}[Bordemann-Waldmann \cite{BW}]
On the K\"ahler manifold $(M,\omega,J)$, there exists a star product of Wick type $*_{\omega,J}$.
\end{theorem}

Their proof is constructive in the sense that the $C_j$'s defining $*_{\omega,J}$ can be obtained recursively.
The product $F*_{\omega,J}H$ for $F,H \in C^{\infty}(M,\C)$ is described by the formula 
$$F*_{\omega,J}H:=FH+i\nu \Lambda^{\alpha \bar{\beta}}\partial_\alpha F \partial_{\bar{\beta}} H + \nu^2 \Lambda^{\alpha_1\bar{\beta}_1}\Lambda^{\alpha_2\bar{\beta}_2}(\nabla^2 F)_{\alpha_1\alpha_2}(\nabla^2 H)_{\bar{\beta}_1\bar{\beta}_2}+O(\nu^3),$$
where the above expression is written in local holomorphic coordinate $\{z_{\alpha}\}$ and their conjugates
$\{\overline{z_{\beta}}\}$.

On a closed K\"ahler manifold, one observe that for $F,H\in C^{\infty}(M,\R)$
$$\int_M [F,H]_{*_{\omega,J}} \omega^n=\frac{\nu^2}{4}i\Omega^{\J}_J(\Lr_{X_F}J,\Lr_{X_H}J) + O(\nu^3),$$
where $\Omega^{\J}_J$ is the symplectic form introduced in Subsection \ref{subsect:JMomega} on the space $\J(M,\omega)$ of almost complex structure compatible with $\omega$. Using the equivariant moment map property of the Hermitian scalar curvature of Equation (\ref{eq:scalmoment}), we give the first non-trivial term of a trace density for $*_{\omega,J}$. It was already computed by Karabegov \cite{Kara}.

\begin{prop}[Karabegov \cite{Kara}]
Let $Scal(J)$ be the scalar curvature of the K\"ahler manifold $(M,\omega,J)$.
If $F,H \in C^{\infty}(M)$, then $\rho:= 1+ \frac{\nu}{4}Scal(J)$ satisfies
$$\int_M (F*_{\omega,J}H-H*_{\omega,J}F)\rho \dvol \equiv 0 \textrm{ mod } O(\nu^3).$$

Consequently, $*_{\omega,J}$ is closed up to order $2$ in $\nu$ if and only if the K\"ahler manifold $(M,\omega,J)$ is of constant scalar curvature.
\end{prop}

\begin{proof}
For $F,H \in C^{\infty}(M)$ :
$$\int_M (F*_{\omega,J}H-H*_{\omega,J}F)\rho \omega^n = i\frac{\nu^2}{4} \left(\Omega^{\J}_J(\Lr_{X_F}J,\Lr_{X_H}J)+\int_M \{F,H\}Scal(J)\dvol\right)+O(\nu^3).$$
Using the moment map equation (\ref{eq:scalmoment}), we see the right hand side is in $O(\nu^3)$.
\end{proof}

Now, we fix the complex structure of $(M,\omega,J)$ and let vary the K\"ahler form inside $\MOm$. One defines the Calabi functional
$$\cal : \MOm \rightarrow \R : \omega_{\phi} \mapsto \int_M (Scal^{\nabla^{\phi}})^2 \frac{\omega_{\phi}^n}{n!},$$
where as before $\nabla^{\phi}$ is the Levi-Civita connection of $g_{\phi}$. Critical points of $\cal$ are called {\bf extremal K\"ahler metrics}.

\begin{prop}[Calabi \cite{cal}]
Extremal K\"ahler metrics in $\MOm$ form a submanifold of $\MOm$ whose connected component are $H_0(M,J)$-orbit.
\end{prop}

When a K\"ahler metric $\omega$ is of constant scalar curvature, then it is extremal and any metrics in its $H_0(M,J)$-orbit is of constant scalar curvature. In terms of star product, it translates into the following :

\begin{cor}
Let $(M,\omega,J)$ be a closed K\"ahler manifold. Consider the star products $*_{\tilde{\omega},J}$ for $\tilde{\omega}\in \MOm$.
Then, the $\tilde{\omega}\in \MOm$ such that $*_{\tilde{\omega},J}$ is closed up to order $2$ form a submanifold of $\MOm$ whose connected components are $H_0(M,J)$-orbits.
\end{cor}

\subsubsection{Fedosov's star products $*_{\nabla,\nu\chi}$}

Here, we consider Fedosov's star products $*_{\nabla,\nu\chi}$ on a symplectic manifold $(M,\omega)$ built using the data of a symplectic connection and a closed $2$-form $\nu\chi\in \nu\Omega^2(M)\subset \nu\Omega^2(M)[[\nu]]$.

For $F,H \in C^{\infty}(M)$, one has
$$F*_{\nabla,\nu\chi}H=FH+\frac{\nu}{2}\{F,H\} + \nu^2 \left(\frac{1}{2}\chi(X_F,X_H)+ k'\Lambda^{i_1j_1}\Lambda^{i_2j_2}(\nabla^2F)_{i_1i_2}(\nabla^2H)_{j_1j_2}\right) + O(\nu^3).$$
The $*_{\nabla,\nu\chi}$-commutator writes 
$$[F,H]_{*_{\nabla,\nu\chi}}=\nu \{F,H\} + \nu^2 \chi(X_F,X_H)+O(\nu^3).$$
And one can check that the trace density for $*_{\nabla,\nu\chi}$ writes $\rho=1+\nu \Lambda^{ij}\chi_{ij}+O(\nu^2)$.

We recall a moment map construction from Donaldson \cite{Donald2}. Assume now that $\chi$ is non-degenerate (i.e. a symplectic form). Consider $\M:=\diff_0(M)$ the identity component of the group of diffeomorphisms of $M$. Then the tangent space $T_f\M=\Gamma(f^*TM)$. The symplectic form $\Omega^{\M}$ on $\M$ is defined by 
$$\Omega^{\M}_f(U,V):= \int_M \chi_{f(x)}(U_x,V_x)\dvol,$$
for $f\in \M$ and $U,V \in \Gamma(f^*TM)$. The group $\ham(M,\omega)$ acts on the right on $\diff_0(M)$ preserving $\Omega^{\M}$. The fundamental vector fields of this action are $(X_F)^{*\M}_f:=f_* X_F$ for $F\in C^{\infty}_0(M)$.
Define the map $\mu:\M\rightarrow C^{\infty}(M)$ by
$$\mu(f) \dvol= - (f^*\chi)\wedge \dsubvol.$$ 
Donaldson \cite{Donald2} shows this map is a moment map, that is 
$$\frac{d}{dt}|_0 \int_M F\mu(f_t)\dvol = \Omega^{\M}_{f_0}((X_F)^{*\M},V),$$
for $F\in C^{\infty}_0(M)$, $f_t$ a smooth path in $\diff_0(M)$ and $V:=\frac{d}{dt}|_0 f_t\in \Gamma(f_0^*TM)$.

In conclusion, when $\chi$ is non-degenerate, the trace density for $*_{\nabla,\nu f^*\chi}$, where $f\in \diff_0(M)$ writes $\rho^{*_{\nabla,\nu f^*\chi}}:= 1+\nu 2\mu(f)+ O(\nu^2)$.

\end{document}